\documentclass[12pt]{amsart}

\usepackage{amssymb}
\usepackage{amsmath}
\usepackage{latexsym}
\usepackage{amsfonts}
\usepackage{enumerate}
\usepackage{mathtools}
\usepackage[mathscr]{euscript}
\usepackage{eqlist}
\usepackage{amsthm}
\usepackage{verbatim} 
\usepackage{cite}
\usepackage[T1]{fontenc}

\usepackage{todonotes}

\usepackage{enumitem} 

\DeclareMathAlphabet{\mathpzc}{OT1}{pzc}{m}{it}
\theoremstyle{plain} 
\newtheorem{theorem}{Theorem}[section] 
 
\newtheorem{corollary}[theorem]{Corollary}

\newtheorem{problem}[theorem]{Problem}

\theoremstyle{plain}

\theoremstyle{definition}
\newtheorem{definition}[theorem]{Definition}
\newtheorem{example}[theorem]{Example}

\theoremstyle{remark}

\newcommand{\nnatural}{\mathbb{N}}

\renewcommand{\restriction}{\mathord{\upharpoonright}}

\newcommand{\cA}{{\mathcal A}}
\newcommand{\cB}{{\mathcal B}}

\newcommand{\cI}{{\mathcal I}}

\newcommand{\cP}{{\mathcal P}}

\newcommand{\oo}{\omega}

\newcommand{\ee}{\varepsilon}

\newcommand{\MaxLen}{{\mathrm MaxLen}}
\newcommand{\ThickPlus}{${\mathit Thick}^+$}
\newcommand{\StronglyPorous}{\mathcal{SP}}

\makeatletter
{\newcount\@hour}
{\newcount\@minute}
\def\timenow{\@hour=\time \divide\@hour by 60
\number\@hour:
  \multiply\@hour by 60 \@minute=\time
  \global\advance\@minute by -\@hour
  \ifnum\@minute<10 0\number\@minute\else
  \number\@minute\fi}
\def\ctimenow{\hfil{\tt \jobname.tex, \today~Time: \timenow }\hfil}
      \let\@oddfoot\ctimenow\let\@evenfoot\ctimenow
\makeatother

\author{Pawe\l{} Klinga}
\address[]{
Pawe\l{} Klinga\\
University of Gda\'nsk \\
Institute of Mathematics \\
Wita Stwosza 57,
80--952 Gda\'nsk, Poland \\
}
\email[P.~Klinga]{pawel.klinga@ug.edu.pl}
\author{Andrzej Nowik}
\address[]{
Andrzej Nowik\\
University of Gda\'nsk \\
Institute of Mathematics \\
Wita Stwosza 57,
80--952 Gda\'nsk, Poland \\
}
\email[A.~Nowik]{andrzej.nowik@ug.edu.pl}
\author{Anna W\k{a}sik}
\address[]{
Anna W\k{a}sik\\
University of Gda\'nsk \\
Institute of Mathematics \\
Wita Stwosza 57,
80--952 Gda\'nsk, Poland \\
}
\email[A. W\k{a}sik]{anna.wasik@phdstud.ug.edu.pl}
\begin{document}
\title[
    $\sigma$-Porosity...
]{
    $\sigma$-Porosity of Certain Ideals
}
\keywords{ideals, $\sigma$-porous sets, Van der Waerden ideal, piecewise syndetic sets}

\subjclass[2020]{03E15, 40A05}

\begin{abstract}
We investigate the $\sigma$-porosity of certain known ideals of subsets of natural numbers. Porosity is a notion of smallness in metric spaces that is stronger than nowhere density. Analogously, $\sigma$-porosity is a strengthening of meagerness. In this paper, we verify which ideals are $\sigma$-porous.
\end{abstract}
\maketitle

\section{Preliminaries}
\medskip

\subsection{Notation}
By $\omega$ we understand the set of naturals starting from 0 and $2^\omega$ its powerset. We identify each subset of $\omega$ with a binary sequence through the characteristic function. Fix $s \in 2^{<\omega}$ (i.e. a finite binary sequence). Then $[s] = \{x\in 2^\omega\colon s \subseteq x\}$ is the element of base in the product topology of $2^\omega$. The space $2^\omega$ itself will be referred to as the Cantor space.

Let $t$ be an element of $2^\omega$ and $j\in \omega$. By $t \restriction_{j}$ we mean the restriction of the sequence $t$ to $j = \lbrace 0, \ldots, j-1\rbrace$. If $s \in 2^{<\omega}$ then by  $s \frown t$ we denote the concatenation of $s$ followed by $t$. For $c \in \{0,1\}$ and $l \in \omega$, by $c^l$ we  denote a constant sequence $u \in 2^\omega$ such that $u_i =c$ for $i=0, 1,... , l-1$.

\medskip

\subsection{Ideals} \label{subsection} An ideal is a nonempty family $\cI \subset 2^{\omega}$ closed under taking subsets and finite unions, i.e.
\begin{enumerate}
	\item $\emptyset \in \cI$
	\item $\forall_{A,B}\:\:A \in \cI \land B\subset A \implies B \in \cI$
	\item $\forall_{A,B}\:\:A \in \cI \land B\in \cI \implies A\cup B \in \cI$
\end{enumerate}
An ideal is proper if $\oo\notin \cI$, or, equivalently, $\cI \neq 2^\omega$. Unless stated otherwise, we assume all ideals are proper. This implies a certain intuition, that an ideal is a family of \textit{small} subsets of $\oo$, in some sense.

A filter is a dual concept, i.e. it is closed under taking supersets and finite intersections. Therefore we understand a filter as a family of \textit{large} subsets of $\oo$. Additionally, a maximal (under inclusion) filter is called an ultrafilter. Finally, we say that an ideal is a maximal ideal if its dual filter, i.e. $\lbrace \omega\setminus A\colon A \in \cI\rbrace$, is an ultrafilter.

Let us proceed to specific ideals.

\begin{definition} Let $(a_n)_{n\in\oo}$ be a sequence of positive reals such that $\sum_{n\in \oo} a_n = \infty$. Then
$$ \cI_{(a_n)} = \left\{ A \subset \omega \colon \sum_{n\in A} a_n < \infty \right\} $$
is its corresponding summable ideal.
\end{definition}
Recall that an ideal is tall if every infinite subset of $\omega$ contains an infinite subset belonging to the ideal. As it turns out, summable ideals are tall if and only if $a_n\to 0$. In our paper we assume all summable ideals are tall.

\begin{definition}
Let $E_n$ denote the family of subsets of naturals which do not contain any arithmetic sequence of length $n$. Put
$$\mathcal{W} = \bigcup_{n} E_n.$$
$\mathcal{W}$ is an ideal called the van der Waerden ideal.
\end{definition}

	Let us consider certain notions of largeness, namely \textit{thickness}, \textit{syndecity} and \textit{piecewise syndecity}. They have been discussed in numerous papers concering ergodic Ramsey theory and topology, for instance \cite{beige},  \cite{BE}, \cite{BE1}, \cite{ZU}.
	 
\begin{definition}
	Let $T \subseteq \omega$. We say that $T$ is thick if
	$$\forall_{p \in \omega} \exists_{n \in \omega} \quad \{n, n+1, n+2, ..., n+p \} \subset T. $$
\end{definition}

\noindent In other words, $T$ is thick if it contains arbitarily long intervals. We will denote the family of thick sets as $Thick$.

\begin{definition}
A set $S \subseteq \omega$  is syndetic if
$$ \exists_{p \in \omega} \forall_{a \in \omega} \quad \{a, a+1, a+2, ..., a+p\} \cap S \neq \emptyset. $$
\end{definition}

\noindent Intuitively, a syndetic set  can be thought of as an interval with gaps, where each gap is of length bounded by a constant $p$. The family of syndetic sets is denoted by $\mathcal{S}$. Observe that a set $T$ is thick if and only if $T \cap S \neq \emptyset $ for all syndetic sets $S$.

\begin{definition} We say that set $P$ is piecewise syndetic if
$$ \exists_{T \in Thick \atop S \in \mathcal{S}} \quad  P=T \cap S.$$
\end{definition}
\noindent Equivalently, $P$ is piecewise syndetic if it consists of arbitrarily long intervals with gaps, where the gaps are of length bounded by some constant $p$. The family of piecewise syndetic sets will be denoted by $\mathcal{PS}$.

While the aforementioned definitions relate to \textit{large} sets, we can consider their \textit{small} counterparts by taking the complements of families.

By \ThickPlus we understand the family of sets which are not thick. We follow the usual notation of $\cI^+ = \omega \setminus \cI$. Notice that \ThickPlus is not an ideal. To observe this fact, consider sets of even and odd numbers.

On the other hand, it is easy to verify that the complement of $\mathcal{PS}$ forms an ideal, which we refer to as the Brown ideal, denoted as $\mathcal{B}$. This name is derived from Brown's lemma which states that whenever we partition a piecewise syndetic set into finitely many subsets, at least one of those subsets must also be piecewise syndetic. It was established by Brown \cite{Br} in the context of locally finite semigroups.

Obviously, any thick set is also piecewise syndetic. Furthermore, it is a well-known fact that any piecewise syndetic set contains arbitrarily long arithmetic sequences. Using contraposition, we can establish the following inclusions: $\mathcal{W} \subset \mathcal{B} \subset$ \ThickPlus.

Finally, we define the ideal $\mathcal{I}_u$. 

\begin{definition}
Let $A \subset \omega$ and $n \in \omega\setminus\lbrace 0 \rbrace$. Put 
$$S_n(A) = \max_{m \in \omega}| A \cap \{m, m+1, ..., m+n\}|,  $$
\noindent where $|X|$ means the cardinality of a set $X$. It is known that the limit $\bar{u}(A) =  \lim_{n \to \infty} \frac{S_n(A)}{n}$ exists. The family
$$\mathcal{I}_u = \{A \subset \omega: \bar{u}(A) = 0 \} $$
\noindent is called the ideal of uniform density zero sets.

	For more details about the ideal of uniform density zero sets, see \cite{BFMS}. It can be easily proven that $\mathcal{I}_u \subset$ \ThickPlus.

\end{definition} 
\subsection{Porosity} Porosity is a metric concept of smallness, similar to nowhere density. The concept is stronger (i.e. the family of porous sets is coarser), but unlike nowhere density, it requires a metric space $X$. Porous sets were introduced in 1967 by Dol\v{z}enko \cite{D} and developed by Zaj\'{i}\v{c}ek \cite{Z1}. Intuitively, porostity  is associated with the size of holes in the set.
 
\begin{definition}
\label{definition:porocity}
We say that a set $E \subseteq X$ is porous if
$$\exists_{0 < \alpha < 1 \atop \ee_0 > 0} 
\forall_{0 < \ee \leq \ee_0 \atop x\in X} \exists_{y \in X} \quad
B(y, \alpha \cdot \ee) \subseteq B(x, \ee) \setminus E.$$
\end{definition}

A set is called $\sigma$-porous if it is a countable union of porous sets. 

Notice that the collection of $\sigma$-porous sets on the real line forms a
$\sigma$-ideal strictly contained in the intersection of the $\sigma$-ideal of meager sets and 
the $\sigma$-ideal of Lebesgue null sets.

  The notion of $\sigma$-porosity has been discussed in multiple papers, 
see for instance \cite{Z2} and \cite{PZ}. Notice that in \cite{Z2} the author considers two kinds of porosity: $\sigma$-upper (weaker property) and $\sigma$-lower porosity (stronger property) and by Proposition 2.2 from \cite{Z2} the Definition \ref{definition:porocity} is in fact equivalent to $\sigma$-lower porosity.

As we consider ideals defined on $2^\omega$, we will use the following metric (in fact an ultrametric):

\medskip
 
$$\rho(s, t) = \begin{cases}
     2^{-\min\lbrace i \in \omega\colon s(i) \not= t(i)\rbrace} & \text{if }s \not= t,\\
     0      & \text{if }s = t,
\end{cases}$$

\noindent where $s_n$ and $t_n$ are binary sequences.

We will be examining the $\sigma$-porosity of certain ideals in the sense of product topology which is induced by $\rho$.

In the final section of the paper we will step away from ideals and discuss another definition of porostity -- a classical approach to this concept which we will translate onto the Cantor space $2^\omega$.

Notice that the notion of porosity from the Definition \ref{definition:porocity}
can be reformulated in the combinatorial way (another but similar look on this topic can be
found in \cite{DBLP:journals/jla/GonzalezHM17}):
\begin{definition}\label{very-porous-warunek-kombinatoryczny}
A set $E\subseteq 2^\omega$ is porous if 
  $$\exists_{\substack{M \in \omega\\ K \in\omega}} \forall_{\substack{t\in 2^{<\omega}\\ |t| > M}}
  \exists_{s \supseteq t} |s| \leq |t| + K \wedge E \cap [s] = \emptyset.$$
\end{definition}

\begin{definition}\label{df:n-porous-sets}
Following \cite{DBLP:journals/jla/GonzalezHM17} we say that a set $A \subseteq 2^\omega$
is $n$-porous if for every $s \in 2^{<\omega}$ there is a $t \in 2^{<\omega}$,
$|t| = n$ such that $[s\frown t]\cap A = \emptyset$. By $\StronglyPorous_n$
we denote the $\sigma$-ideal generated by $n$-porous sets.
\end{definition}

The notation $\StronglyPorous_n$ comes from ``strongly porous'' --  indeed, by \cite{DBLP:journals/jla/GonzalezHM17} it is a stronger notion, i.e. every $\StronglyPorous_n$ set is $\sigma$-porous. The reverse inclusion does not hold.
\section{Summable ideals}

\begin{theorem}
	No summable ideal $\cI_{(a_n)}$ is a $\sigma$-porous subset of $2^\omega$.
\end{theorem}

\begin{proof}
	Let us assume a contrario that there exists a summable ideal $\cI_{(a_n)}$, where $a_n \geqslant 0$, $\sum_{n\in\omega} a_n = \infty$ and $a_n \to 0$, that is a $\sigma$-porous set. Therefore it can be represented as the following union:
	$$\cI_{(a_n)} = \bigcup_{n\in\omega} F_n,$$
	where each $F_n$ is porous. This means that for each $n\in\omega$ there exist $\alpha_n \in (0, 1)$ and $\varepsilon_n\in (0,1)$ such that if $0 < \varepsilon < \varepsilon_n$ and $x_0 \in 2^\omega$, then there exists $y_0$ such that
	$$B(y_0, \alpha_n \varepsilon) \subseteq B(x_0, \varepsilon) \setminus F_n.$$
	
	We shall construct inductively a sequence $(s_n)$ such that $s_n \in 2^{<\omega}$ and $s_n \subseteq s_{n+1}$, as well as two companion sequences $(M_n), (N_n)$ of naturals. For the completeness of the construction we put $s_{-1} = \emptyset$ and
	$M_{-1} = 0$. Since $\lim_{n\to\infty} a_n = 0$ we can pick
        $N_n < M_n$ such that the following conditions hold:
	\begin{enumerate}
		\item
		$N_n > \max\lbrace n, M_{n-1}, \log_2(\frac{1}{\varepsilon_n}) \rbrace$;
		\item
		$M_n - N_n > \log_2(\frac{1}{\alpha_n})$;
		\item
		$\sum_{i=N_n}^{M_n} a_i < \frac{1}{2^n}$;
	\end{enumerate}
	
	We define $x_0^{(n)} = s_{n-1} \frown 0^\infty$ and	$\varepsilon_n^\prime = 2^{-N_n}$. Since $N_n > \log_2(\frac{1}{\varepsilon_n})$, then $-N_n < \log_2(\varepsilon_n)$, hence
	$-\log_2(2^{N_n}) < \log_2(\varepsilon_n)$, therefore $\varepsilon_n^\prime < \varepsilon_n$. Out of our assumption of porosity, it implies that there exists $y_0^{(n)}$ such that
	$$B(y^{(n)}_0, \varepsilon_n^\prime \alpha_n) \subseteq B(x^{(n)}_0, \varepsilon_n^\prime) \setminus F_n.$$
	Since $\rho(y^{(n)}_0, x^{(n)}_0) < \varepsilon_n^\prime$, then $x^{(n)}_0 \restriction_{N_n} = y^{(n)}_0 \restriction_{N_n}$. We will now check that
	$$[y^{(n)}_0 \restriction_{M_n}] \subseteq B(y^{(n)}_0, \varepsilon_n^\prime \alpha_n).$$
	It is true that $M_n - N_n > \log_2(\frac{1}{\alpha_n})$, hence	$N_n - M_n < \log_2(\alpha_n)$ and then $2^{N_n - M_n} < \alpha_n$. Therefore $2^{-M_n} < 2^{-N_n} \cdot \alpha_n = \varepsilon_n^\prime \alpha_n$. Fix $a \in [y^{(n)}_0 \restriction_{M_n}]$. We have that $a \restriction_{M_n} = y^{(n)}_0 \restriction_{M_n}$, so $\rho(a, y^{(n)}_0) \leq 2^{-M_n} < \varepsilon_n^\prime \alpha_n$ and therefore $a \in B(y^{(n)}_0, \varepsilon_n^\prime \alpha_n)$, which proves the inclusion.
	
	Let us put $s_n = y^{(n)}_0 \restriction_{M_n}$. We have that $s_{n-1} \subseteq x^{(n)}_0$, 
	$x^{(n)}_0 \restriction_{N_n} = y^{(n)}_0 \restriction_{N_n}$ and $s_n = y^{(n)}_0 \restriction_{M_n}$. At the same time $|s_{n-1}| < N_n < M_n$, hence
	$s_{n-1} \subseteq s_n$. It is also true that $N_n > n$, hence $|s_n| \to \infty$. We define
	$z = \bigcup_{n\in\omega} s_n$. We have	$z \restriction_{M_n} = y^{(n)}_0 \restriction_{M_n}$ and
	$$[y^{(n)}_0 \restriction_{M_n}] \subseteq B(y^{(n)}_0, \varepsilon_n^\prime \alpha_n) \subseteq B(x^{(n)}_0, \varepsilon_n^\prime) \setminus F_n.$$
	Therefore $z\not\in \bigcup_{n\in\omega} F_n = \cI_{(a_n)}$.
	
	We will now show that $\sum_{z(i) = 1} a_i < \infty$, which will be the desired contradiction. Let us take a closer look at the placement of the crucial numbers:
	\begin{align*}
		& M_{-1} = |s_{-1}| = 0 < N_0 < M_0 = |s_0| < \\
		<\:\:&  N_1 < M_1 = |s_1| < N_2 < M_2 = |s_2| < \dots
	\end{align*}
	
	Let us notice that for any $n > 0$ we have $z\restriction_{(M_n, N_{n+1})} = 0 \restriction_{(M_n, N_{n+1})}$. It follows from the fact that
	$$x^{(n+1)}_0 \restriction_{N_{n + 1}} = y^{(n+1)}_0 \restriction_{N_{n + 1}} \subseteq 
	y^{(n+1)}_0 \restriction_{M_{n + 1}} = s_{n+1} \subseteq z,$$
	and since (following the definition of $x^{(n+1)}_0$) it is true that $x^{(n+1)}_0 \restriction_{[|s_n|, \infty)} = 0 \restriction_{[|s_n|, \infty)}$ (and $|s_n| = M_n$), then ultimately we obtain indeed $z\restriction_{(M_n, N_{n+1})} = 0 \restriction_{(M_n, N_{n+1})}$.
	Hence
	$$\sum_{i > M_0 \atop z(i) = 1} a_i \leqslant \sum_{n = 1}^{\infty} \sum_{i \in [N_n, M_n]} a_i 
	\leqslant \sum_{n = 1}^{\infty} \frac{1}{2^n} < \infty,$$
	which proves that $z \in \cI_{(a_n)}$, a contradiction.
\end{proof}

\section{Van der Waerden ideal}

\begin{theorem} \label{thm:W-is-sigma-porous}
	$\mathcal{W}$ is a $\sigma$-porous subset of $2^\omega$.
\end{theorem}

\begin{proof}
	By $ASq_n$ let us denote the set of all arithmetic sequences of length $n$:
	\[
	ASq_n = \lbrace \lbrace a_1, a_1 + r, \ldots, a_1 + (n - 1) r \rbrace \colon
	a_1 \in \omega \wedge r > 0
	\rbrace.
	\]
	Then
	$$E_n = \lbrace A \subseteq \omega \colon \forall_{F \in ASq_n} F \not\subseteq A \rbrace.$$
	Notice that for each $n$ this set is closed. We will check that it is porous. Fix $n$. Take $\varepsilon_0 = 1$ and define $\alpha = \frac{1}{2^{n+2}}$.
	
	Fix $\varepsilon < \varepsilon_0$ and a ball $B(x_0, \varepsilon)$. Then for any $N\geqslant -\log_2\varepsilon$ we have $[x_0\restriction N] \subseteq B(x_0, \varepsilon)$, since then $\frac{1}{2^N} \leqslant \varepsilon$. Let us take the \textit{first} such $N$, i.e.
	$$ \frac{1}{2^N} \leqslant \varepsilon < \frac{1}{2^{N-1}}.$$
	
	Let us define $y_0 = x_0 \restriction (N + 1) \frown 1^\infty$. We will show that it is true that
	$$B(y_0, \varepsilon \cdot \alpha ) \subseteq B(x_0, \varepsilon) \setminus E_n.$$
	First, let us notice that $B(y_0, \varepsilon \cdot \alpha) \subseteq B(x_0, \varepsilon)$.
	Take $a \in B(y_0, \varepsilon \cdot \alpha)$. We have $\rho(a, x_0) \leqslant \rho(a, y_0) + \rho(y_0, x_0) < \varepsilon \cdot \alpha + \frac{\varepsilon}{2} \leqslant \varepsilon$ (since $\alpha \leqslant \frac{1}{2}$), which proves the inclusion.
	
	Now we will show that $ a \notin E_n $. Notice that
	$$\rho(y_0, a) < \alpha\cdot\varepsilon < \frac{1}{2^{n+2}} \cdot \frac{1}{2^{N-1}} = \frac{1}{2^{N + n + 1}},$$
	therefore $a \restriction (N + 1 +n) = y_0 \restriction (N + 1 + n)$, since if there existed $i < N + 1 + n$ such that $a(i) \not= y_0(i)$ then $i \geqslant \min\lbrace j\colon a(j) \not= y_0(j) \rbrace$ and $\frac{1}{2^{N + n + 1}} \leqslant 2^{-i} \leqslant \rho(a, y_0)$, which would be a contradiction. We also have $a \restriction [N + 1, N + n] = y_0 \restriction [N + 1, N + n] = 1 \restriction [N + 1, N + n] \in ASq_n$. It shows that $a \not\in E_n$.
\end{proof}

Next we show that while $\mathcal{W}$ is $\sigma$-porous, it is not in $\StronglyPorous_n$. This fact places the van der Waerden ideal between two distinct concepts of porosity.

\begin{theorem}\label{thm:W-is-not-SP-n}
	$\mathcal{W}$ is not in the family $\StronglyPorous_n$ for any $n$.
\end{theorem}

\begin{proof}
Suppose a contrario that $W \subseteq \StronglyPorous_N$ for some $N$.
Then there exist: a sequence of sets $(A_k)$ such that
$W \subseteq \bigcup_{k = 0}^\infty A_k$ and
for each $s \in 2^{<\omega}$ and $k \in \omega$
a sequence $t(s; k) \in 2^{<\omega}$, $|t(s; k)| \leq N$ such
that $[s \frown t(s; k)] \cap A_k = \emptyset$.

We will construct an increasing sequence of finite binary sequences $(s_m)$ by the following description. Put $s_0 = 0^{N+1}$. Having constructed $s_m$, define $s_m^* = s_m \frown t(s_m, m)$. Then define $M(m) = \max( |s_m^* | + 1, N + 1)$ and put $s_{m+1} = s_m^* \frown 0^{M(m)}$. Next, define $x_0 = \bigcup_{m=0}^\infty s_m$.

It is evident that $x_0 \not\in \bigcup_{k = 0}^\infty A_k$.
Let us check that $x_0 \in W$, in fact, we check that
$\hat{x_0} = \lbrace k \colon x_0(k) = 1\rbrace \in W$.
We will prove that $\hat{x_0}$ does not contain any 
increasing arithmetic sequence of size $N + 1$.

Indeed, suppose by way of contradiction that 
there exist $a_1$ and $r > 0$ such that
$R = \lbrace a_1, a_1 + r, \ldots, a_1 + N \cdot r\rbrace \subseteq \hat{x_0}$. 
We show that the whole $R$ is contained in some
integer interval $[|s_m|, |s_m^*|) \cap \omega$.
It suffices to show that if a triple $\lbrace a, a + r, a + 2r\rbrace$
($r > 0$) is contained in $\bigcup_{m=0}^\infty [|s_m|, |s_m^*|) \cap \omega$,
then it is contained in some $[|s_m|, |s_m^*|) \cap \omega$.
So suppose that $a, a + r \in [0;|s_m^*|) \cap \omega$
and $a + 2r \in [|s_l|, |s_l^*|) \cap \omega$ for some $l > m$.
Then $r \leq |s_m^*|$, thus 
$a + 2r = (a + r) + r \leq |s_m^*| + |s_m^*|$, while
$|s_l| \geq |s_{m+1}| = |s_m^*| + M(m) \geq 2 |s_m^*| + 1$,
which is a constradiction.

On the other hand, suppose that $a \in [0;|s_m^*|) \cap \omega$
and $a + r, a + 2r \in [|s_l|, |s_l^*|) \cap \omega$
for some $l > m$. Then
$r = (a + 2r) - (a + r) \leq |s_l^*| - |s_l| \leq N$,
but $a = (a + r) - r \geq |s_l| - N \geq |s_{m+1}| - N = |s_{m}^*| + M(m) - N$
but $M(m) \geq N$, hence $a \geq |s_{m}^*|$, which is a contradiction
with $a \in [0;|s_m^*|) \cap \omega$.

Hence the whole sequence $R$ is contained in some 
interval $[|s_m|, |s_m^*|) \cap \omega$, which is impossible
since $|[|s_m|, |s_m^*|) \cap \omega| \leq N$.
\end{proof}

Notice that it is possible to strengthen Theorem \ref{thm:W-is-sigma-porous} using similar techniques, resulting in the following.

\begin{theorem}\label{thm:Thick+-is-sigma-porous}
	\ThickPlus is $\sigma$-porous.
\end{theorem}

\begin{proof}
	Let us start by decomposing \ThickPlus into a countable family:
	$$Thick^+ = \bigcup_{N=1}^\infty Thick^+_N,$$
	where
	$$Thick^+_N = \{ A\subseteq \omega \colon 	\forall_{\substack{I\subset \omega\\ |I| = N}} I \not\subseteq A \}. $$
	For each $N$ we have that $Thick^+_N$ is a closed set. It is easy to see that $Thick^+_N$ is porous. This follows from the observation that for each $t\in 2^{<\omega}$
	$$ [ t\frown 1^N ] \cap Thick^+_N = \emptyset. $$	
\end{proof}

\ThickPlus being $\sigma$-porous yields strong consequences. As $\sigma$-porosity is a notion of smallness, it is obviously hereditary: every subset of a $\sigma$-porous set is also $\sigma$-porous.
\begin{corollary}\label{cor:ideal-subset-of-thick-plus-is-sigma-porous}
For any ideal $\cI \subseteq$ \ThickPlus, $\cI$ is $\sigma$-porous.
\end{corollary}

\begin{corollary}
	The following ideals are $\sigma$-porous: $\mathcal{W}$, $\mathcal{B}$, $\mathcal{I}_u$.
\end{corollary}

Notice that we alse have the reverse implication (this was suggested by Jacek Tryba):
\begin{theorem}\label{sugestia-Jacka-Tryby}
  For any ideal $\cI$ if $\cI$ is $\sigma$-porous then
$\cI \subseteq$ \ThickPlus.
\end{theorem}

\begin{proof}
Let $\cI \subseteq 2^\omega$ be such ideal that it is not included in \ThickPlus. Pick $A \in \cI \setminus {\mathit Thick}$. It means that $A$ contains intervals of an arbitrary length. Since $\cI$ is an ideal, we can assume without the loss of generality that $A = \bigcup_{n} I_n$, where $(I_n)$ are pairwise disjoint intervals, $\max I_n < \min I_{n+1}$ and $|I_n| \geq n$. Define
\[
P = \lbrace x \in 2^\omega \colon x \restriction_{\omega \setminus A} =
\underline{0} \restriction_{\omega \setminus A}\rbrace,
\]
where $\underline{0}$ is a sequence constantly equal to 0.

Notice that $P \subseteq \cI$. Therefore it is sufficient to show that $P$ is not $\sigma$-porous. Let us assume otherwise: $P = \bigcup_{n\in\omega} E_n$, where $(E_n)$ is porous for each $n$. For every set $E_n$ let $M_n, K_n$ be such natural numbers that exist for $E_n$ due to the condition from Definition \ref{very-porous-warunek-kombinatoryczny}.

Let us outline a few simple observations:
\begin{enumerate}[label=(\arabic*)]
\item For any $t \in 2^{<\omega}$ we have $[t] \cap P \not=\emptyset$
  iff $\forall_{k < |t|} k \not\in A \implies t(k) = 0$.
  
\item If $[t] \cap P \not=\emptyset$ then
  for each $k$, $[t \frown \langle 0^k \rangle] \cap P \not= \emptyset$.

  This follows immediately from the previous observation.

\item If we have $t \in 2^{<\omega}$ and $k\in\omega$, and any
  extension $t$ to the sequence $s$ of length $|s| + k$ is such that
  $[s] \cap P \not= \emptyset$ then we say that the set $P$ is \underline{k-tight} in the basic set $[t]$

\item\label{observation4}
   For any $t \in 2^{<\omega}$ such that $[t] \cap P \not=\emptyset$
   and for any $k\in\omega$ we can extend
   $t$ to an arbitrarily long $s$ such that the set $P$ is k-tight in the basic set $[s]$.
   
   Indeed, for a given $t$ we can add a sufficiently long sequence of zeros, so that the new sequence, $s$, fulfills $|s| = \min I_{k_1}$ for some $k_1 > k$
   (notice that $|s| - 1$ is the last index of the sequence $s$). Such $s$ meets the requirements, since $|I_{k_1}| \geq k_1 > k$.
\end{enumerate}

We construct the sequence $t_k \in 2^{<\omega}$ by induction, so that $[t_k] \cap P \not= \emptyset$. We extend $t_{k-1}$ with zeros so that the length of this sequence was greater than $M_k$ and then, using observation \ref{observation4} we extend it towards the sequence $t_k^\prime$ so that $P$ is $K_k$-tight in the basic set $[t_k^\prime]$.

Then we use the assumption regarding the set $E_k$ and extend $t_k^\prime$ towards such sequence $t_k$ that $[t_k] \cap P \not= \emptyset$ and $[t_k] \cap E_k = \emptyset$. Define: $x_0 = \bigcup_{k} t_k$. Due to the closeness of the set $P$ we obtain $x_0 \in P$ but $x_0 \not\in \bigcup_{k} E_k$, which is a contradiction.
\end{proof}

\section{Classical porous sets}

By $\MaxLen(A)$ let us denote the maximum length of an open interval $J \subseteq A$ (and zero if no such interval exists).

\begin{definition}
A set $E\subset \mathbb{R}$ is upper porous if for every $x \in E$ we have
$$\lim\sup_{\varepsilon\to 0^+} \frac{\MaxLen((x - \varepsilon, x + \varepsilon) \setminus E)}{\varepsilon} > 0.$$
\end{definition}

We will translate this condition onto the Cantor space $2^\omega$. An interval of the form $(x - \varepsilon, x + \varepsilon)$ will be replaced by a basic set $[x\restriction_n]$. Analogously, $\MaxLen(Z)$ will denote the maximum diameter of a basic set contained in $Z \subseteq 2^\omega$ (or 0 if $Z$ has an empty interior). That is, for ${\mathit int}(Z) \not=\emptyset$ we have
\[
  \MaxLen(Z) = \max \lbrace 2^{-|s|} \colon [s] \subseteq Z, s \in 2^{<\omega} \rbrace.
\]
Therefore the definiton of upper porosity of a set $A$ in the Cantor space can be stated as follows: 
\begin{definition}
	A set $A \subset 2^\omega $ is upper porous if for every $x\in A$
\begin{equation}\label{klasyczna-porowatosc-z-limsup-na-Cantor}
	\limsup_{n\to\infty} \frac{\MaxLen([x\restriction_n] \setminus A)}{2^{-n}} > 0.
\end{equation}
\end{definition}
The expression $\limsup_{n\to\infty} a_n > 0$ means exactly the following: 
$$\exists_{\eta > 0} \exists_{n_k} \forall_{k \in \nnatural} a_{n_k} \geq \eta$$
where $n_k$ is a subsequence. Without the loss of generality we may assume that $\eta$ is of the form $2^{-K}$. Therefore we can reconstruct the condition (\ref{klasyczna-porowatosc-z-limsup-na-Cantor})
into
\begin{equation}\label{klasyczna-porowatosc-bez-limsup-na-Cantor}
  \forall_{x \in A} \exists_{K \in \nnatural} \exists_{n_k} \forall_{k\in\nnatural} \MaxLen([x\restriction_{n_k}]\setminus A) 
  \geq {2^{-(n_k + K)}};
\end{equation}
which is equivalent to
\begin{equation}\label{klasyczna-porowatosc-bez-limsup-na-Cantor-2}
  \forall_{x \in A} \exists_{K \in \nnatural} \exists_{n_k} \forall_{k\in\nnatural}
  \exists_{s\in 2^{<\omega}} [s] \subseteq [x\restriction_{n_k}]\setminus A 
  \wedge {2^{|s|} \geq 2^{-(n_k + K)}}.
\end{equation}

Since  $[s] \subseteq [x\restriction_n]$, then $x\restriction_n \subseteq s$,
therefore we obtain the final condition:
\begin{equation}\label{klasyczna-porowatosc-na-Cantor}
  \forall_{x \in A} \exists_{K \in \nnatural} \exists_{n}^\infty 
  \exists_{s\in 2^{<\omega}} x\restriction_n \subseteq s \wedge [s] \cap A = \emptyset
  \wedge |s| \leq n + K.
\end{equation}
Therefore by transitiong it onto the Cantor space, the definition of upper porosity obtains the following final form.

\begin{definition}\label{klasyczna-porowatosc-warunek-kombinatoryczny}
A set $E\subseteq 2^\omega$ is upper porous if
  $$\forall_{x_0 \in E} \exists_{K\in \omega} \exists_{n}^\infty
  \exists_{s \in 2^{<\omega}} |s| \leq n + K \wedge s \supseteq x_0 \restriction_n \wedge [s] \cap E = \emptyset.$$
\end{definition}

It is known that the concepts of porosity and upper porosity are different. For completeness, we will name an example of a perfect set in the Cantor space which illustrates the difference between these two definitions.

\begin{example}
Let $A = \lbrace n^2, n^2 + 1 \colon n = 1,2,3,\ldots\rbrace$ and put $P = \lbrace x \in 2^\omega \colon x \restriction_A \equiv\underline{0}\restriction_A \rbrace$.

We will show that $P$ is upper porous and not $\sigma$-porous (and so obviously not porous). 

In order to show upper porosity, pick any $x_0 \in P$. Then for every $n \in \omega$ we define $s\in 2^{<\omega}$ as $s = x_0 \restriction_n \frown \langle 1 \rangle$. Then $|s| = n + 1$ and $[s] \cap P = \emptyset$ (since $s(n) = 1$). Therefore the combinatorial condition (\ref{klasyczna-porowatosc-warunek-kombinatoryczny}) is met.

We proceed to showing that $P$ is not $\sigma$-porous. Define the following ideal:
\[
\cI = \lbrace Z \cup F \colon Z \subseteq \omega\setminus A \wedge
F \in {\mathsf Fin}\rbrace.
\]
It is easy to verify that $\cI = P \oplus {\mathsf Fin}$, where $\cA \oplus \cB = \lbrace A \oplus B \colon A \in \cA, B \in \cB\rbrace$,
and $A \oplus B$ denotes the sum modulo $2$ in $2^\omega$. Notice that $\cI$ is not included in \ThickPlus, since $\omega \setminus \lbrace n^2, n^2 + 1\colon n \in \nnatural\rbrace \in \cI$ and such set contains intervals of an arbitrary length.

It follows from Theorem \ref{sugestia-Jacka-Tryby} that the set $P \oplus {\mathsf Fin}$ is not $\sigma$-porous. Since this set is a union of countably many shifts of the set $P$ (and shifting the set preserves its porosity), therefore $P$ also cannot be $\sigma$-porous.

\end{example}

Let us end with the following:
\begin{problem}
Find a suitable 
(like Corollary \ref{cor:ideal-subset-of-thick-plus-is-sigma-porous} and Theorem \ref{sugestia-Jacka-Tryby})
characterization of all ideals which are $\sigma$-upper porous.
\end{problem}

\bibliography{biblio_poro}
\bibliographystyle{Abbrv}

\end{document}